\documentclass[a4paper, 11pt, reqno]{amsart}

\usepackage{amssymb, a4, mathrsfs, color, cite, rotating}

\newcommand\GreenL{\mathscr{L}}
\newcommand\GreenR{\mathscr{R}}
\newcommand\GreenH{\mathscr{H}}
\newcommand\GreenD{\mathscr{D}}
\newcommand\GreenJ{\mathscr{J}}
\newcommand\trop{\mathbb{T}}
\newcommand\olt{{\overline{\mathbb{T}}}}
\newcommand\ft{\mathbb{FT}}
\newcommand\tropn{\trop^{n \times n}}
\newcommand\oltn{\olt^{n \times n}}
\newcommand\ftn{\ft^{n \times n}}
\newcommand{\br}[2]{\langle {#1} \mid {#2} \rangle}
\newcommand{\rb}[1]{(\br{r_1}{#1}, \dots, \br{r_k}{#1} )}

\newtheorem{theorem}{Theorem}[section]
\newtheorem{lemma}[theorem]{Lemma}
\newtheorem{corollary}[theorem]{Corollary}
 
\newtheorem{proposition}[theorem]{Proposition}

\theoremstyle{definition}

\theoremstyle{definition}

\begin{document}

\title[Tropical matrix duality and Green's $\GreenD$ relation]{Tropical matrix duality and Green's $\GreenD$ relation}

\maketitle

\begin{center}

    CHRISTOPHER HOLLINGS\footnote{Christopher Hollings' current address:
Mathematical Institute, University of Oxford, \\ 24--29 St Giles', Oxford OX1 3LB. Email \texttt{christopher.hollings@maths.ox.ac.uk}.} and
 MARK KAMBITES\footnote{Email \texttt{Mark.Kambites@manchester.ac.uk}.}

    \medskip

    School of Mathematics, \ University of Manchester, \\
    Manchester M13 9PL, \ England.

\date{\today}
\thanks{}

\end{center}

\begin{abstract}
We give a complete description of Green's $\GreenD$ 
relation for the multiplicative semigroup of all $n \times n$ tropical 
matrices. Our main tool is a new variant on the \textit{duality} between 
the row and column space of a tropical matrix (studied by Cohen, Gaubert 
and Quadrat and separately by Develin and Sturmfels). Unlike the existing 
duality theorems, our version admits a converse, and hence gives a necessary
and sufficient condition 
for two tropical convex sets to be the row and column space of a matrix. 
We also show that the matrix duality map induces an isometry (with respect 
to the Hilbert projective metric) between the projective row space and 
projective column space of any tropical matrix, and establish some
foundational results about Green's other relations.
\end{abstract}
\maketitle

\medskip

Tropical algebra (also known as max-plus algebra or max-algebra) is the
algebra of the real numbers (sometimes augmented with $-\infty$ and/or
$+\infty$) when equipped
with the binary operations of addition and maximum. It has traditional
applications in a wide range of subjects, such as combinatorial optimisation
and scheduling problems \cite{Butkovic03}, analysis of discrete event
systems \cite{MaxPlus95}, control theory \cite{Cohen99}, formal language
and automata theory \cite{Pin98, Simon94} and combinatorial/geometric group
theory \cite{Bieri84}. More recently, exciting connections have emerged with
algebraic geometry \cite{Bergman71, Mikhalkin05, RichterGebert05}; these
have also led to new applications in areas such as phylogenetics \cite{Eriksson05}
and statistical inference \cite{Pachter04}. The first
detailed axiomatic study was conducted by
Cuninghame-Green \cite{CuninghameGreen79} and this theory has been
developed further by a number of researchers (see \cite{Baccelli92,
Heidergott06} for surveys).

Since many of the problems which arise in application areas are naturally
expressed in terms of (max-plus) linear equations, much of tropical algebra
is concerned with matrices.
Many researchers have had cause to
prove \textit{ad hoc} results about the
\textit{multiplication} of tropical matrices; there has also been considerable
attention paid to certain special questions such as Burnside-type problems for
semigroups of tropical matrices \cite{dAlessandro04,Gaubert06,Pin98,Simon94}.
Surprisingly, though, there has been relatively little systematic study of
these semigroups, and little is known about their abstract algebraic structure.
In particular, there has been until recently no understanding of the semigroup
of all matrices of a given size over the tropical semiring, comparable with
the classical theory of the general linear group or full matrix semigroup
over a field. The detailed study of this semigroup was recently initiated
by Johnson and the second author \cite{K_tropicalgreen} and independently
by Izhakian and Margolis \cite{Izhakian10}.

Green's relations \cite{Green51,Clifford61} are five equivalence relations
($\mathcal{L}$, $\mathcal{R}$, $\mathcal{H}$, $\mathcal{D}$ and $\mathcal{J}$)
and three pre-orders ($\leq_\GreenR$, $\leq_\GreenL$ and $\leq_\GreenJ$)
which can be defined upon any semigroup, and which encapsulate the structure
of its principal left, right and two-sided ideals and maximal subgroups. They
are
amongst the most powerful tools for understanding the structure of
semigroups and monoids, and play a key role in almost every aspect of
modern semigroup theory.
The relations $\leq_\GreenR$, $\leq_\GreenL$, $\mathscr{L}$, $\mathscr{R}$
and $\mathscr{H}$ are easily described for matrix semigroups over arbitrary
semirings with identity and zero. For example, two matrices are
$\mathscr{R}$-related exactly if they have the same column space; see
\cite{K_tropicalgreen} or Section~\ref{sec_greens} below for a more
detailed explanation.

The relations $\mathscr{D}$, $\mathscr{J}$ and $\leq_\GreenJ$, however, are
rather more subtle. In the classical case of a finite dimensional full matrix
semigroup over a field, it is well known that $\GreenD$ and $\GreenJ$ coincide
and encapsulate the concept of \textit{rank}, with the pre-order $\leq_\GreenJ$
corresponding to the obvious order on ranks. Johnson and the
second author \cite{K_tropicalgreen} showed that $\GreenD$ and $\GreenJ$ also
coincide for the $2 \times 2$ tropical matrix semigroup, with the $\GreenJ$-class
of a matrix being determined by the \textit{isometry type} of its row space
(or equivalently, its column space).

The main aim of the present paper is to give a complete description of Green's $\mathscr{D}$
relation for the full matrix semigroup of arbitrary finite dimension over the
tropical semiring. Specifically, we show that two matrices are $\mathscr{D}$-related
exactly if their row spaces (or equivalently, their column spaces) are
isomorphic as semimodules (Theorems~\ref{thm_d} and \ref{thm_dtrop} below).

While our main result develop those from the $2 \times 2$ case,
proved in \cite{K_tropicalgreen}, the methods used here are entirely
different. The
results of \cite{K_tropicalgreen} were obtained naively by direct algebraic
manipulations, which are unenlightening even in two dimensions, and quickly
become impractical in higher dimensions. Here, our approach is geometric,
with the main tool being the phenomenon
of \textit{duality} between the row space and column space of a tropical
matrix. It is well known that there
is a natural and canonical bijection between the row space and column space
of a given tropical matrix which preserves certain aspects of their structure.
Specifically, Cohen, Gaubert and Quadrat \cite{Cohen04} have showed that it is an
antitone lattice isomorphism, while Develin and Sturmfels \cite{Develin04}
have proved that it induces a combinatorial isomorphism of (Euclidean) polyhedral
complexes. Here we shall employ a slight variation of the duality theorem, which states
that the duality map is in a certain algebraic sense an \textit{anti-isomorphism}
of semimodules (Theorem~\ref{thm_duality}). The important thing for our
purpose is that this version of duality admits a converse: any two finitely generated convex
sets which are anti-isomorphic must be the row and column space of a
tropical matrix (Theorem~\ref{thm_dgeometry}). As a corollary of our duality
theorem, we also observe that the duality map is an isometry (with respect
to the Hilbert projective metric) between the row and column spaces of any
tropical matrix. We believe that these results
are likely to be of independent interest.

In addition to this introduction, the paper comprises six sections.
Section~\ref{sec_prelims} provides a brief summary of the necessary
background material from tropical algebra and geometry. Section~\ref{sec_duality}
introduces and proves our new version of the tropical duality theorem, and discusses
briefly its relationship to existing results. Section~\ref{sec_greens}
recalls the definitions of Green's relations, and establishes some basic
properties of Green's relations in finite dimensional full matrix semigroups over the
tropical semirings. Section~\ref{sec_converseduality}
proves our converse to the duality theorem. Section~\ref{sec_greend}
applies in the preceding results to describe Green's $\GreenD$ relation in
these semigroups. Finally, Section~\ref{sec_remarks} contains some remarks
on questions remaining open and the potential application of our methods in
wider contexts.

Because of a recent proliferation of applications, tropical mathematics
is now of interest to a broad range of researchers with radically different
motivations and backgrounds. This article is likely to be of interest to
many of these people and, in addition, to abstract semigroup theorists with
no experience of tropical mathematics. For this reason, we have endeavoured to
keep the article self-contained by minimising the use of specialist
terminology, notation and machinery, and by including elementary proofs for
foundational results where feasible. In so doing we crave the
indulgence of specialists in this area, who may feel that parts of the paper
could be expressed more concisely.

\section{Preliminaries}\label{sec_prelims}

The \textit{finitary tropical semiring} $\ft$ is the semiring (without
additive identity) consisting of the real numbers under the operations
of addition and maximum.
The \textit{tropical semiring} $\mathbb{T}$ is the finitary tropical semiring
augmented with an extra element $-\infty$, which acts as a zero for addition
and an identity for maximum. The \textit{completed tropical semiring} $\olt$
is the tropical semiring augmented with an extra element $\infty$, which
acts as a zero for both maximum and addition, save that
$(-\infty) + \infty = \infty + (-\infty) = -\infty$. 
These structures share many of the properties of fields, with addition
and maximum playing the roles of field multiplication and field addition
respectively; for this reason we write $a \oplus b$ for $\max(a,b)$ and
$a \otimes b$ or just $ab$ for $a + b$. In particular, note that $\otimes$
distributes over $\oplus$ and that both operations are commutative and
associative. Hence, they give rise to a natural, associative multiplication
on matrices over $\olt$.

We extend the usual order $\leq$ on the reals to a total order on $\olt$
in the obvious way, with $-\infty < x < \infty$ for all $x \in \mathbb{R}$,
noting that $a \oplus b = a$ exactly if $b \leq a$ and that $\olt$ under
$\leq$ is a complete lattice.
The semirings $\ft$ and $\olt$ admit a natural order-reversing
involution, $x \mapsto -x$. In $\ft$ this involution obviously
distributes over $\otimes$ (so $-(xy) = (-x)(-y)$), but more caution is
required in $\olt$ where $- (\infty (-\infty)) \neq (-\infty) \infty$. The involution
has the following elementary property, which is obvious in $\ft$ but needs
to be verified by case analysis in $\olt$.

\begin{proposition}\label{prop_completeineq}
For any $a, b, c \in \olt$ we have
$$a b \leq c \iff a (-c) \leq (-b).$$
\end{proposition}
\begin{proof}
We suppose that $a b \leq c$ and show that $a (-c) \leq (-b)$, the
reverse implication being dual. If $a$, $b$ and $c$ are all finite then
the claim is immediate. If $a = -\infty$ then $a(-c) = -\infty$ so the
desired inequality holds. If $a = \infty$ then for $a b \leq c$ we must
have either $b = - \infty$ or $c = \infty$; in both cases the desired
inequality again holds. If $b = -\infty$ then the claim is immediate,
while if $b = \infty$ then we must have either $a = -\infty$ or $c = \infty$,
again guaranteeing the claim. Finally, if $c = \infty$ then the claim is
immediate, while if $c = -\infty$ then we must have $a = -\infty$ or
$b = -\infty$, placing us in a case we have already considered.
\end{proof}

For $S \in \lbrace \ft, \mathbb{T}, \olt \rbrace$ we are interested in
the space $S^n$
of \textit{(affine tropical) vectors}. We write $x_i$ for the $i$th component
of a vector $x \in S^n$. We extend $\oplus$ and $\leq$
to $S^n$ componentwise, so that $(x \oplus y)_i = x_i \oplus y_i$, 
and $x \leq y$ exactly if $x_i \leq y_i$
for all $i$.
Sometimes we wish to stress that a space $S^n$ is composed of row vectors or
of column vectors, in which case we write
$S^{1 \times n}$ or $S^{n \times 1}$
respectively. We define a scaling action of $S$ on $S^n$ by
$$\lambda (x_1, \ldots, x_n) = (x_1+\lambda, \ldots, x_n+\lambda)$$
for each $\lambda, x_1, \dots, x_n \in S$. From affine tropical $n$-space we obtain
\textit{projective tropical $(n-1)$-space} (denoted $\mathbb{PFT}^{n-1}$,
$\mathbb{PT}^{n-1}$ or $\overline{\mathbb{PT}}^{n-1}$ as appropriate) by identifying two
vectors if one is a tropical multiple of the other by an element of $\ft$.  

An ($S$-linear) \textit{convex set} in $S^n$ is a subset closed under
$\oplus$ and scaling by elements of $S$, that is, a linear subspace of $S^n$.
If $X \subseteq S^n$ then the ($S$-linear) \textit{span} or \textit{convex hull} of $T$
is the set of all vectors which can be written as tropical linear
combinations (with scaling and vector addition $\oplus$ as defined
above) of finitely many vectors from $X$ with coefficients drawn from $S$. It is
easily seen that if $X \subseteq \ft^n$ then the $\ft$-linear span of $X$ 
is the intersection with $\ft^n$ of the $\mathbb{T}$-linear span of $X$.
Similarly, if $X \subseteq \mathbb{T}^n$ then the $\mathbb{T}$-linear span
of $X$ is the intersection with $\mathbb{T}^n$ of the $\olt$-linear span of
$X$.
Each convex set $X$ induces a subset of the corresponding projective space,
termed the \textit{projectivisation} of $X$.
Convex sets in tropical space have a very interesting structure, which has
been extensively studied
(see for example \cite{Block06,Butkovic07,Cohen04,Cohen05,Develin04,Gaubert06b,Gaubert07,Joswig07})
but is still not fully understood.

We define a scalar product operation $\olt^n \times \olt^n \to \olt$
on completed tropical $n$-space $\olt^n$ by setting
$$\br{a}{b} = \max \lbrace \lambda \in \olt \mid \lambda a \leq b \rbrace.$$
The existence of such a maximum is easily verified by case analysis.
This operation is a \textit{residual} operation in the sense of
residuation theory \cite{Blyth72}; it has been quite
extensively used in tropical mathematics (see for example \cite{Cohen04}).
We shall need some elementary properties of this operation.

\begin{proposition}\label{prop_bracketequiv}
For any $n \geq 1$ and $x, y \in \olt^n$,
$$\br{x}{y} = -\left(\bigoplus_{i=1}^n\{x_i \otimes (-y_i) \}\right).$$
\end{proposition}
\begin{proof}
First, notice that if $x_i = -\infty$ or $y_i = \infty$ then
$x_i (-y_i) = -\infty$ and $\lambda x_i \leq y_i$ for all $\lambda$, so
the $i$th component makes no contribution to the maximum in the
statement and no difference to the maximum in the definition of
$\br{x}{y}$. If \textit{all} components are like this, then it is readily
verified that both $\br{x}{y}$ and the right hand side in the statement take the
value $\infty$, so the proposition holds. Otherwise, we may discard any
such components and seek to prove the claim only in the case that
$x_i \neq -\infty$ and $y_i \neq \infty$ for all $i$.

If $x_i = \infty$ [respectively $y_i = -\infty$] then since we are assuming
$y_i \neq \infty$ [$x_i \neq -\infty$] we have $x_i (-y_i) = \infty$ so
that
$$-\bigoplus_{i=1}^n\{x_i (-y_i) \} = -\infty.$$
We also have
$\lambda x_i > y_i$ for all $\lambda > -\infty$, so that $\langle x \mid y \rangle = -\infty$.

Thus, we may assume that $x, y \in \ft^n$. Now by definition we have
$\lambda x \leq y$ if and only if
$\lambda x_i \leq y_i$ for all $i$. By Proposition~\ref{prop_completeineq}
this is true exactly if $\lambda (- y_i) \leq (-x_i)$, that is, if and
only if
$-\lambda \geq x_i - y_i$ for all $i$. It follows that $\max \lbrace x_i - y_i \rbrace$
is the smallest $-\lambda$ such that $\lambda x \leq y$, and so
$- \max \lbrace x_i - y_i \rbrace$ is the largest $\lambda$ such that
$\lambda x \leq y$, which is by definition $\langle x \mid y \rangle$.
\end{proof}

\begin{proposition}\label{prop_bracketsignchange}
Let $x, y \in \olt^n$. Then $\langle x \mid y \rangle = \langle -y \mid -x \rangle$
\end{proposition}
\begin{proof}
By Proposition~\ref{prop_completeineq} we have
$\lambda (-y_i) \leq (-x_i)$ for all $i$ if and only if
$\lambda x_i \leq y_i$ for all $i$. Thus, $\lambda (-y) \leq (-x)$
if and only if $\lambda x \leq y$, so
$$\langle x \mid y \rangle = \max \lbrace \lambda \mid \lambda x \leq y \rbrace = \max \lbrace \lambda \mid \lambda (-y) \leq (-x) \rbrace = \langle -y \mid -x \rangle.$$
\end{proof}

\begin{proposition}\label{prop_orderbracket}
For any $x, y \in \olt^n$, we have $x \leq y$ if and only if
$\br{x}{y} \geq 0$.
\end{proposition}
\begin{proof}
If $x \leq y$ then $0x \leq y$ so by the definition of $\br{x}{y}$ we have
$0 \leq \br{x}{y}$. Conversely, if $\langle x \mid y \rangle \geq 0$ then
$x = 0x \leq \br{x}{y} x \leq y$.
\end{proof}

\begin{proposition}\label{prop_coeffs}
Let $n \geq 1$ and $a, r_1, \dots, r_k \in \olt^n$ be such that 
$a$ lies in the convex hull of $r_1, \dots, r_k$. Then
$$a = \bigoplus_{i=1}^k \langle r_i | a \rangle r_i.$$
\end{proposition}
\begin{proof}
For $1 \leq j \leq n$, let $b_j$ be the $j$th component of the right hand
side; we must show that each $a_j = b_j$.
Since $a$ lies in the convex hull of $r_1, \dots, r_k$, it can be written
in the form
$$a = \bigoplus_{i=1}^k \lambda_i r_i$$
for some scalars $\lambda_i \in \olt$. Then for $1 \leq j \leq n$ we
have $a_j = \bigoplus_{i=1}^k \lambda_i (r_i)_j$, so there is an $i$
such that $a_j = \lambda_i (r_i)_j$. Now $\lambda_i r_i \leq a$
so by the definition of $\langle r_i \mid a \rangle$ we have
$\lambda_i \leq \langle r_i | a \rangle$, and so
$a_j = \lambda_i (r_i)_j \leq \langle r_i | a_i \rangle (r_i)_j \leq b_j$.

Conversely, it follows from the definition of $\br{r_i}{a}$ that
$\br{r_i}{a} r_i \leq a$ for all $i$. Thus, $\br{r_i}{a} (r_i)_j \leq a_j$
for all $i$ and $j$, and hence $b_j \leq a_j$ for all $j$.
\end{proof}

We define a distance function on $\olt^n$ by $d_H(x,y) = 0$ if $x$ is a
finite scalar multiple of $y$, and
$$d_H(x,y) = - (\langle x \mid y \rangle \otimes \langle y \mid x \rangle)$$
otherwise.
It is easily verified that the map $d_H$ is invariant under scaling $x$
or $y$ by finite tropical scalars, and hence well-defined on
$\overline{\mathbb{PT}}^{n-1}$. In fact, it is well known that $d_H$ is
an extended metric (that is, a metric which is permitted to take the
value $\infty$) on projective space. We call it the
\textit{(tropical) Hilbert projective metric}.

\begin{proposition}
$d_H$ is an extended metric on $\overline{\mathbb{PT}}^{n-1}$.
\end{proposition}
\begin{proof}
For the triangle inequality, suppose $x, y, z \in \olt^n$. If $x$ and $y$
are finite scalar multiples, or if $y$ and $z$ are finite scalar multiples
then the triangle inequality is trivially satisfied. Otherwise, by
the definition of the bracket we have $\br{x}{y} x \leq y$ and
$\br{y}{z} y \leq z$. Thus,
$$(\br{x}{y} + \br{y}{z}) x = \br{y}{z} \br{x}{y} x \leq \br{y}{z} y \leq z$$
which by the definition of $\br{x}{z}$ means that $\br{x}{y} + \br{y}{z} \leq \br{x}{z}$.
A symmetrical argument shows that $\br{z}{y} + \br{y}{x} \leq \br{z}{x}$, and
combining these we have
\begin{align*}
d_H(x,y) + d_H(y,z) &= - (\br{x}{y}+ \br{y}{x}) - (\br{y}{z} + \br{z}{y}) \\
&= - (\br{x}{y}+ \br{y}{z} + \br{z}{y} + \br{y}{x}) \\
&\geq - (\br{x}{z} + \br{z}{x}) \\
&= d_H(x,z).
\end{align*}
The other conditions are readily verified from the definition.
\end{proof}

Notice that any map $\theta$ from a subspace of $\olt^p$ to a subspace of
$\olt^q$ which either preserves the bracket ($\br{x}{y} = \br{\theta(x)}{\theta(y)}$)
or reverses the bracket ($\br{x}{y} = \br{\theta(y)}{\theta(x)}$) will
preserve the Hilbert projective metric. We call such maps \textit{orientation-preserving}
and \textit{orientation-reversing}, respectively.

\section{Duality}\label{sec_duality}

In this section we begin by providing a brief introduction to tropical
matrix duality, We introduce the notion of an \textit{anti-isomorphism}
of semimodules, establish some basic properties of anti-isomorphisms,
and show that the matrix duality map is an anti-isomorphism.

Let $S = \ft$, $S = \trop$ or $S = \olt$. Let $p, q \geq 1$ and $A$ be a
$p \times q$ matrix over $S$, with rows $A_1, \dots, A_p \in S^{1 \times q}$ and columns $B_1, \dots, B_q \in S^{p \times 1}$.
We define the ($S$-linear) \textit{row space} $R_S(A)$ of $A$ to be the
$S$-linear convex hull of the vectors $A_1, \dots, A_p$.
Dually, the \textit{($S$-linear) column space} $C_S(A)$ is the
$S$-linear convex hull of the vectors $B_1, \dots, B_q$. Since we are most
often interested in the case $S = \olt$, we shall for brevity
write $C(A)$ and $R(A)$ for $C_\olt(A)$ and $R_\olt(A)$ respectively. 

Now treating $A$ as a matrix over $\olt$, we define a map
$\theta_A : R(A) \to C(A)$ by
$$\theta_A(x) = A (-x)^T = \bigoplus_{i = 1}^q (-x_i) B_i$$
where the second equality is immediate from the definition of matrix
multiplication. Notice that, again just using the definition of matrix
multiplication,
the $i$th component of $\theta_A(x)$ is
$$\bigoplus_{j=1}^q \lbrace A_{ij} + (- x_j) \rbrace,$$
which by Proposition~\ref{prop_bracketequiv} is exactly
$- \langle A_i \mid x \rangle$. 

Similarly, we define $\theta_A' : C(A) \to R(A)$ by
$$\theta_A'(y) = (-y)^T A = \bigoplus_{j=1}^p (-y_j) A_j.$$
There is an obvious duality between $\theta_A'$ and $\theta_A$ via the
transpose map; indeed for every $y \in C(A)$ we have $y^T \in R(A^T)$ and
$\theta_A'(y) = (\theta_{A_T}(y^T))^T$. From this, or directly, we may
deduce that the $j$th component of $\theta_A'(y)$ is 
$- \langle B_j \mid y \rangle$.

The map $\theta_A$, which we shall call the \textit{duality map} of $A$,
was studied (in a rather more general axiomatic setting, of which
the completed tropical semiring $\olt$ is a special case) by Cohen,
Gaubert and Quadrat \cite{Cohen04}, who established that it is an antitone
isomorphism of lattices. Its restriction to $\ft$ has also been considered
by Develin and Sturmfels \cite{Develin04}
who observed that it preserves the Euclidean polytope structure of the row
space. We shall need a slight strengthening of these results; we make no claim
of originality in respect of this, since the stronger form can be deduced from
\cite{Cohen04} and is probably essentially known to experts in the field.
However, since \cite{Cohen04} is not very accessible to non-specialists, and
we believe this paper will have a rather broader readership, we include a
direct, elementary combinatorial proof.

We begin with the following elementary property of the duality map.
\begin{proposition}\label{prop_dualitybijection}
For any matrix $A$ over $\olt$, the maps $\theta_A$ and $\theta_A'$ are
mutually inverse bijections between $R(A)$ and $C(A)$. If the entries of
$A$ are all finite then $\theta_A$ and $\theta_A'$ restrict to mutually
inverse bijections between $R_\ft(A)$ and $C_\ft(A)$.
\end{proposition}
\begin{proof}
Suppose $A$ is a $p \times q$ matrix, and let $B_1, \dots, B_q$ be the columns
of $A$. We claim that $\theta_A \circ \theta_A'$ is the identity
function on $C(A)$. Indeed, given $c \in C(A)$, by the observations
above we have
$$\theta_A'(c) = (- \br{B_1}{c}, \dots, - \br{B_q}{c}).$$
Now using the definition of $\theta_A$ we have
$$\theta_A(\theta_A'(c)) = \bigoplus_i (- (- \br{B_i}{c})) B_i = \bigoplus_i \br{B_i}{c} B_i = c,$$
where the last equality is guaranteed by Proposition~\ref{prop_coeffs}
because $c$ lies in the convex hull of $B_1, \dots, B_q$. A dual argument shows
that $\theta_A' \circ \theta_A$ is the identity function on $R(a)$, which
completes the proof that $\theta_A$ and $\theta_A'$ are mutually inverse
bijections between $R(A)$ and $C(A)$.

Now suppose all entries of $A$ are finite. In this case, it is immediate
from the definition that $\theta_A$ maps finite vectors to finite vectors, so
it maps $R_\ft(A)$ into $C_\ft(A) \cap \ft^p$. But by our observations in
Section~\ref{sec_prelims} we have $C_\ft(A) = C(A) \cap \ft^p$, so $\theta_A$
maps $R_\ft(A)$ into $C_\ft(A)$. A dual argument shows that $\theta_A'$ maps
$C_\ft(A)$ into $R_\ft(A)$. Since $\theta_A$ and $\theta_A'$ are mutually
inverse bijections, it follows that they restrict to mutually inverse
bijections between $C_\ft(A)$ and $R_\ft(A)$.
\end{proof}

Notice that the duality maps $\theta_A$ and $\theta_A'$ are defined for
matrices over $\ft$ or $\olt$, but do not make sense for matrices over
$\mathbb{T}$ because they depend upon the
involution $x \mapsto -x$. Indeed, in the special case that $A$ is the
$1 \times 1$ matrix with single entry $0$, the duality map is exactly
this involution. The core of the proof of our duality theorem is the
following elementary property of the bracket operation.

\begin{lemma}\label{lemma_changecoords}
Suppose $n \geq 1$ and $a, b, r_1, \dots r_k \in \olt^n$. Then
$$\langle a \mid b \rangle \leq \br{\rb{a}}{\rb{b}}$$
with equality provided $a$ and $b$ are contained in the $\olt$-linear
convex hull of $r_1, \dots, r_k$.
\end{lemma}
\begin{proof}
Let
$x = \rb{a} \text{ and } y = \rb{b}.$
We first show that $\br{a}{b} \leq \br{x}{y}$. Firstly, if
$\br{x}{y} = \infty$ then there is nothing to prove. Next, suppose
$\br{x}{y} = -\infty$. By Proposition~\ref{prop_bracketequiv} there is
an $i$ such that $x_i (-y_i) = \infty$, which means either
\begin{itemize}
\item[(1)] $x_i = \infty$ and $y_i \neq \infty$; or
\item[(2)] $x_i \neq -\infty$ and $y_i = -\infty$.
\end{itemize}

In case (1), since $\br{r_i}{b} = y_i \neq \infty$,
there exists $j$ such that $(r_i)_j \neq -\infty$ and $b_j \neq \infty$.
But since $\br{r_i}{a} = x_i = \infty$ we must have $a_j = \infty$.
But now $a_j (-b_j) = \infty$ which by Proposition~\ref{prop_bracketequiv}
again ensures that $\br{a}{b} = -\infty$.

In case (2) we have $\br{r_i}{a} = x_i \neq -\infty$ and
$\br{r_i}{b} = y_i = -\infty$. Applying the same argument as above to
the latter, there is a $j$ such that
either
\begin{itemize}
\item[(2A)] $(r_i)_j = \infty$ and $b_j \neq \infty$; or
\item[(2B)] $(r_i)_j \neq -\infty$ and $b_j = -\infty$.
\end{itemize}
In case (2A), since $(r_i)_j = \infty$ but $\br{r_i}{a} \neq -\infty$ we must
have $a_j = \infty$, whereupon $a_j (-b_j) = \infty$. In case (2B), by
a similar argument, we must have $a_j \neq -\infty$ and again $a_j (-b_j) = \infty$,
which by Proposition~\ref{prop_bracketequiv} ensures that $\br{a}{b} = -\infty$.

Now consider the case in which $\br{x}{y}$ is finite. By
Proposition~\ref{prop_bracketequiv} there is an $i$ such that
$\br{x}{y} = -(x_i (-y_i))$. Since $\br{x}{y}$ is finite, $x_i$ and
$y_i$ must be finite, so we have $\br{x}{y} = y_i - x_i = \br{r_i}{b} - \br{r_i}{a}$.
By the same argument, since $\br{r_i}{b} = y_i$ is finite, there is a $j$
such that $\br{r_i}{b} = b_j - (r_i)_j$ with $b_j$ and $(r_i)_j$ finite,
whereupon
\begin{equation}\label{eq1}
b_j = \br{r_i}{b} + (r_i)_j.
\end{equation}
Also, by the definition of the bracket, we have
\begin{equation}\label{eq2}
a_j \geq \br{r_i}{a} + (r_i)_j.
\end{equation}
Since all terms in \eqref{eq1} and \eqref{eq2} (with the possible exception
of $a_j$ which may be $\infty$) are known to be finite, we may
subtract \eqref{eq1} from \eqref{eq2} to obtain
$$a_j - b_j \geq \br{r_i}{a} + (r_i)_j - (r_i)_j - \br{r_i}{b} = \br{r_i}{a} - \br{r_a}{b} = x_i - y_i.$$
But now by Proposition~\ref{prop_bracketequiv} we have
$$\br{a}{b} = - \bigoplus_{k=1}^n \lbrace a_k (- b_k) \rbrace \leq -(a_j - b_j) \leq -(x_i - y_i) = \br{x}{y}.$$

It remains to show that $\br{x}{y} \leq \br{a}{b}$ under the assumption
that $a$ and $b$ lie in the convex hull of the vectors $r_1, \dots, r_k$.
Under this assumption, Proposition~\ref{prop_coeffs} ensures that
$$a = \bigoplus_{i=1}^k \br{r_i}{a} r_i \text{ and } b = \bigoplus_{i=1}^k \br{r_i}{b} r_i.$$
Suppose $\lambda \in \olt$ is such that $\lambda x \leq y$. Then by definition
$\lambda x_i \leq y_i$ for all $i$, that is, 
$\lambda \br{r_i}{a} \leq \br{r_i}{b}$ for all $i$. Using the compatibility
of the order with $\otimes$ it follows that $\lambda \br{r_i}{a} r_i \leq \br{r_i}{b} r_i$ for all $i$
and hence using the compatibility of the order with $\oplus$ and distributivity
of $\otimes$ over $\oplus$ that
$$\lambda a = \lambda \bigoplus_{i=1}^k \br{r_i}{a} r_i = \bigoplus_{i=1}^k \lambda \br{r_i}{a} r_i \leq \bigoplus_{i=1}^k \br{r_i}{b} r_i = b.$$
We have shown that $\lbrace \lambda \mid \lambda x \leq y \rbrace \subseteq \lbrace \lambda \mid \lambda a \leq b \rbrace$
and so
$$\br{x}{y} = \max \lbrace \lambda \mid \lambda x \leq y \rbrace
\leq \max \lbrace \lambda \mid \lambda a \leq b \rbrace  = \br{a}{b}.$$
\end{proof}

Lemma~\ref{lemma_changecoords} is the key ingredient for our new formulation
of tropical matrix duality:

Let $X \subseteq \olt^n$ and $Y \subseteq \olt^m$ be convex sets. We
say that a function $\theta : X \to Y$ is an \textit{anti-morphism}
if
\begin{itemize}
\item for all $x, y \in X$, we have $\br{x}{y} = \br{\theta(y)}{\theta(x)}$; and
\item for all $x \in X$ and $\lambda \in \ft^n$ we have $\theta(\lambda x) = (-\lambda) \theta(x)$.
\end{itemize}
Notice that an anti-morphism is required to preserve scaling only by
\textit{finite} scalars and that, by Proposition~\ref{prop_orderbracket},
an anti-morphism must be order-reversing.
A bijective anti-morphism is called an \textit{anti-isomorphism}; an
anti-isomorphism is in particular an antitone lattice morphism.
Notice that the inverse of an anti-isomorphism is necessarily an
anti-isomorphism. Two sets are termed \textit{anti-isomorphic} if there is
an anti-isomorphism between them.

\begin{lemma}\label{lemma_composeantis}
Let $S = \olt$ or $S = \ft$, let $X \subseteq S^i$, $Y \subseteq S^j$
and $Z \subseteq S^k$
be convex sets, and suppose $\theta_1 : X \to Y$ and $\theta_2 : Y \to Z$
are anti-isomorphisms between convex sets.
Then the composition $\theta_2 \circ \theta_1$ is a linear isomorphism of
semimodules.
\end{lemma}
\begin{proof}
Since $\theta_1$ and $\theta_2$ are antitone lattice isomorphisms, their
composition is certainly a lattice isomorphism. The fact that $\theta_2 \circ \theta_1$
respects addition now follows from the fact that addition can be defined
in terms of the lattice order in $X$ and $Z$. Indeed, since $X$ is convex, for
any elements $x, y \in X$, $x \oplus y$ is the least upper bound of $x$
of $y$ in the lattice order on $X$. Since $\theta_2 \circ \theta_1$ is a
lattice isomorphism it follows that $\theta_2(\theta_1(x \oplus y))$ is
the least upper bound of $\theta_2(\theta_1(x))$ and $\theta_2(\theta_1(y))$
in the lattice order on $Z$, which since $Z$ is convex is exactly
$\theta_2(\theta_1(x)) \oplus \theta_2(\theta_1(y))$.

Also, for any finite $\lambda$ we
have
$$\theta_2(\theta_1(\lambda x)) = \theta_2((-\lambda) \theta_1(x)) = \lambda \theta_2( \theta_1(x))$$
so the composition $\theta_2 \circ \theta_1$ preserves scaling by finite
scalars. In the case $S = \ft$, this completes the proof.

In the case $S = \olt$, we must show also that $\theta_2 \circ \theta_1$
preserves scaling by $-\infty$ and $\infty$. Let $z_i$ and $z_k$
denote the zero vectors in $\olt^i$ and $\olt^k$ respectively.
Since $X$, $Y$ and $Z$ are convex, each contains the
zero vector of the appropriate dimension. Indeed, each must have the zero
vector as its bottom element, which since $\theta_2 \circ \theta_1$
is a lattice isomorphism means that $\theta_2(\theta_1(z_i)) = z_k$.
Thus we have 
$$\theta_2(\theta_1((-\infty) x)) = \theta_2(\theta_1(z_i)) = z_k = (-\infty) \theta_2(\theta_1(x))).$$
It remains only to show that $\theta_2 \circ \theta_1$ preserves scaling by
$\infty$. Notice that for any vector $x \in \olt^n$ we have
$$\infty x = \sup \lbrace \lambda x \mid \lambda \in \ft \rbrace.$$
The supremum here is by definition taken in $\olt^n$, but if $x$ lies in
some convex set $S \subseteq \olt^n$ then $\infty x$ and $\lambda x$ for
all $\lambda \in \ft$ also lies in $S$, and it follows that we may take
the supremum in $S$. 
Now since $\theta_2 \circ \theta_1$ is a lattice isomorphism of convex sets,
it preserves suprema within $X$. Since it also preserves scaling by finite
scalars, we have:
\begin{align*}
\theta_2(\theta_1(\infty x)) &= \theta_2(\theta_1(\sup \lbrace \lambda x \mid \lambda \in \ft \rbrace)) \\
&= \sup \lbrace \theta_2(\theta_1(\lambda x)) \mid \lambda \in \ft \rbrace \\
&= \sup \lbrace \lambda \theta_2(\theta_1(x)) \mid \lambda \in \ft \rbrace \\
&= \infty \theta_2(\theta_1(x)).
\end{align*}
\end{proof}

\begin{theorem}[Algebraic Duality Theorem]\label{thm_duality}
Let $A$ be a matrix over $\olt$. Then the map $\theta_A : R(A) \to C(A)$
is an anti-isomorphism between $R(A)$ and $C(A)$. If all entries of $A$ are
finite then $\theta_A$ restricts to an
anti-isomorphism between $R_\ft(A)$ and $C_\ft(A)$.
\end{theorem}
\begin{proof}
By Proposition~\ref{prop_dualitybijection}, $\theta_A$ is a bijection
from $R(A)$ to $C(A)$. Suppose $a, b \in R(A)$, and let $A_1, \dots, A_p$ be the rows of $A$. Then by definition, $a$ and $b$ lie in the convex hull of the $A_i$'s.
Using the definition of $\theta_A$, Lemma~\ref{lemma_changecoords} says
exactly that $\br{a}{b} = \br{-\theta_A(a)}{-\theta_A(b)}$. But
by Proposition~\ref{prop_bracketsignchange} we have 
$\br{-\theta_A(a)}{-\theta_A(b)} = \br{\theta_A(b)}{\theta_A(a)}$. 

Next, suppose the columns of $A$
are $B_1, \dots, B_q$ and let $x \in R(A)$ and $\lambda \in \ft$. Then
$-\lambda \in \ft$ and we have
$$\theta_A(\lambda x) = \bigoplus_{i = 1}^q (-(\lambda x)_i) B_i
 = \bigoplus_{i = 1}^q (-\lambda) (- x_i) B_i =
 (-\lambda) \bigoplus_{i = 1}^q (- x_i) B_i = (-\lambda) \theta_A(x).$$
\end{proof}

We shall see later (Theorem~\ref{thm_dgeometry} below)
that the existence of an anti-isomorphism between finitely generated
convex sets $X$ and $Y$ is a sufficient, as well as a necessary, condition
for $X$ and $Y$ to be the row space and column space of a matrix over
$\ft$ or $\olt$.

As a corollary of Theorem~\ref{thm_duality}, we obtain a special case of the 
theorem of Cohen, Gaubert and Quadrat \cite{Cohen04}.
\begin{theorem}[Lattice Duality Theorem \cite{Cohen04}]\label{thm_latticeduality}
Let $A$ be a matrix over $\olt$. Then the duality
maps $\theta_A$ and $\theta_A'$ are mutually inverse antitone lattice
isomorphisms between $R(A)$ and $C(A)$. If $A$ has all entries finite
then $\theta_A$ and $\theta_A'$ restrict to mutually inverse antitone
lattice isomorphisms between $R_\ft(A)$ and $C_\ft(A)$.
\end{theorem}
\begin{proof}
By Proposition~\ref{prop_dualitybijection}, $\theta_A$ and $\theta_A'$ are
mutually inverse bijections between $R(A)$ and $C(A)$, and restrict to
mutually inverse bijections between $R_\ft(A)$ and $C_\ft(A)$ where appropriate.
Hence, it will suffice
to show that $\theta_A$ is order-reversing.
For any $x, y \in R(A)$, by Theorem~\ref{thm_duality}
we have $\langle x \mid y \rangle = \langle \theta_A(y) \mid \theta_A(x) \rangle$.
Thus, using Proposition~\ref{prop_orderbracket} twice, $x \leq y$ if and only
if $\langle x \mid y \rangle = \langle \theta_A(y) \mid \theta_A(x) \rangle \geq 0$,
which holds exactly if $\theta_A(y) \leq \theta_A(x)$.
\end{proof}

Another immediate consequence of Theorem~\ref{thm_duality}, which
does not seem to have been previously noted, is that the duality map
induces an isometry (with respect to the Hilbert metric) between the
projective row space and projective column space of a tropical
matrix.

\begin{theorem}[Metric Duality Theorem]\label{thm_metricduality}
Let $A$ be a tropical matrix.  Then the duality maps $\theta_A$ and
$\theta_A'$ induce mutually inverse isometries
(with respect to the Hilbert
projective metric) between the projectivisations of $R(A)$ and $C(A)$. If
the entries of $A$ are all finite then their restrictions induce mutually
inverse isometries (with respect to the Hilbert projective metric) between
the projectivisations of $R_\ft(A)$ and $C_\ft(A)$.
\end{theorem}
\begin{proof}
By Proposition~\ref{prop_dualitybijection}, $\theta_A$ and $\theta_A'$ are
mutually inverse bijections between $R(A)$ and $C(A)$. By
Theorem~\ref{thm_duality} they map finite scalings to finite scalings, and
hence induce well-defined maps on the respective projective spaces.
Also by Theorem~\ref{thm_duality}, they reverse the bracket operation, and hence
by the observations at the end of
Section~\ref{sec_prelims}, they preserves the distance function.
\end{proof}

\section{Green's Relations}\label{sec_greens}

In this section we briefly recall the definitions of Green's relations;
for a more detailed introduction we refer the reader to one of the
introductory texts on semigroup theoery, such as \cite{Clifford61}. We
then prove some foundational results about Green's relations in tropical
matrix semigroups.

Let $S$ be any semigroup. We denote that $S^1$ the monoid obtained by
adjoining an extra identity element $1$ to $S$. We define a pre-order (a reflexive, transitive
binary relation) $\leq_\GreenR$ on $S$ by $a \leq_\GreenR b$ if there
exists $c \in S^1$ such that $bc = a$, that is, if $a$ lies in the principle
right ideal $bS^1$ generated by $b$, or equivalently, if $aS^1 \subseteq bS^1$.
We define an equivalence relation $\GreenR$
on $S$ by $a \GreenR b$ if $a \leq b$ and $b \leq a$, that is, if $aS^1 = bS^1$.
The pre-order $\leq_\GreenR$ thus induces a partial order on the set of
equivalence classes of $\GreenR$.

Dually, we define $a \leq_\GreenL b$ if $S^1a \subseteq S^1b$, and $a \GreenL b$
if $S^1a = S^1b$. Similarly, we let $a \leq_\GreenJ b$ if $S^1aS^1 \subseteq S^1bS^1$
and $a \GreenJ b$ if $S^1aS^1 = S^1bS^1$. We let $\GreenH$ be the intersection of
$\GreenL$ and $\GreenR$ (so $a \GreenH b$ if $a \GreenL b$ and $a \GreenR b$)
and $\GreenD$ be the smallest equivalence relation containing both $\GreenL$
and $\GreenR$. It is well known and easy to show (see for example \cite{Clifford61})
that $a \GreenD b$ if and only if there exists $c \in S$ with
$a \GreenL c$ and $c \GreenR b$ (or dually, if and only if there exists
$d \in S$ with $a \GreenR d$ and $d \GreenL b$).

We begin with an elementary description of the relations $\GreenR$ and
$\GreenL$ for full matrix semigroups over our tropical semirings, and
indeed over a wider class of semirings. We say that
a commutative semiring $S$ has \textit{local zeros} if for every finite set
$X \subseteq S$ there exists an element $z \in S$ such that $z+x = x$ for
all $x \in X$. The semirings $\ft$, $\tropn$ and $\olt$ all have local zeros,
since given any finite set $X$ if elements it suffices to choose $z$ to be
any element smaller than those in $X$. (In the latter two cases one may of
course always choose $\infty$). The following is a generalisation of a
result proved for $\tropn$ in \cite{K_tropicalgreen}.

\begin{proposition}\label{prop_landr}
Let $A$ and $B$ be elements of a full matrix semigroup over a commutative
semiring with multiplicative identity and local zeros. Then
\begin{itemize}
\item[(i)] $A \leq_R B$ if and only if $C(A) \subseteq C(B)$;
\item[(ii)] $A \GreenR B$ if and only if $C(A) = C(B)$;
\item[(iii)] $A \leq_L B$ if and only if $R(A) \subseteq R(B)$;
\item[(iv)] $A \GreenL B$ if and only if $R(A) = R(B)$;
\item[(v)] $A \GreenH B$ if and only if $C(A) = C(B)$ and $R(A) = R(B)$.
\end{itemize}
\end{proposition}
\begin{proof}
It clearly suffices to show (i). Indeed, (iii) is dual to (i), (ii)
and (iv) follow from (i) and (ii) respectively, and (v) follows
from (ii) and (iv).

Suppose, then, that $A \leq_\GreenR B$, then by definition there is a matrix
$X \in S^{n \times n}$ such that $BX = A$. Now, since
the columns of $BX$ are contained in $C(B)$ it follows that
$C(BX) = C(A) \subseteq C(B)$.

Conversely, suppose $C(A) \subseteq C(B)$. Since the semiring has a
multiplicative identity, every column of $A$ is a linear combination of
a subset of the columns of $A$, and hence lies in $C(A)$ and also in
$C(B)$. Thus, every column of $A$ can be written as a linear combination
of some subset of the columns of $B$. But since the semiring has local
zeros, this means it can be written as a linear combination of all of the
columns of $B$, which means exactly that there exists $X \in S^{n \times n}$ such
that $A = BX$, and so $A \leq_\GreenR B$.
\end{proof}

We next consider the relationship between Green's relations in the respective
full matrix semigroups over the three semirings $\ft$, $\trop$ and $\olt$.
Notice that if $A$ is a matrix over one semiring which is contained in
another, then the row and column space of $A$ depend upon the semiring
over which it is considered. Hence, it is not immediate from
Proposition~\ref{prop_landr} that, for example, two matrices in
$\ft^{n \times n}$ which are $\GreenR$-related in $\olt^{n \times n}$
must also $\GreenR$-related in $\ft^{n \times n}$. However, it transpires
that this is nevertheless the case.

\begin{proposition}\label{prop_rorderinherit1}
Let $A$ and $B$ be matrices in $\ft^{n \times n}$. Then $A \leq_\GreenR B$
in $\ft^{n \times n}$ if and only if $A \leq_\GreenR B$ in $\trop^{n \times n}$.
\end{proposition}
\begin{proof}
Suppose $A \leq_\GreenR B$ in $\tropn$. Then there
is a matrix $P \in \tropn$ such that $A = BP$. Since $B$ and $P$
have finitely many entries, we may choose some finite $\delta \in \ft$
smaller than $b + p - b'$ for every pair of entries $b, b'$ of $B$ and every
finite entry $p$ of $P$. Let $P'$ be obtained from $P$ by replacing every
$-\infty$ entry with $\delta$.

Now let $1 \leq i, j \leq n$. Then
\begin{equation}\label{eq_x}
(BP')_{ij} = \bigoplus_{k=1}^n B_{ik} P'_{kj}
\end{equation}
while
\begin{equation}\label{eq_y}
A_{ij} = (BP)_{ij} = \bigoplus_{k=1}^n B_{ik} P_{kj}.
\end{equation}
Choose $k$ such that $B_{ik} P_{kj}$ is maximum, that is, such that
$B_{ik} P_{kj} = A_{ij}$. Since $A_{ij}$ is finite, we must have
$P_{kj}$ finite, so $P'_{kj} = P_{kj}$ and
$B_{ik} P'_{kj} = B_{ik} P_{kj} = A_{ij}$. Moreover, all the other entries
in the maximum in $\eqref{eq_x}$ are either of the form $B_{ih} P_{hj}$
(which cannot exceed $B_{ik} P_{kj}$ by the assumption on $k$) or of the
form $B_{ih} \delta$ (which cannot exceed $B_{ik} P_{kj}$ by the definition
of $\delta$). Thus, the maximum in $\eqref{eq_x}$ is $A_{ij}$ and we have
$A = BP'$.

The converse is immediate.
\end{proof}

\begin{proposition}\label{prop_rorderinherit2}
Let $A$ and $B$ be matrices in $\trop^{n \times n}$. Then $A \leq_\GreenR B$
in $\trop^{n \times n}$ if and only if $A \leq_\GreenR B$ in $\olt^{n \times n}$.
\end{proposition}
\begin{proof}
Suppose $A \leq_\GreenR B$ in $\olt^{n \times n}$. Then there
is a matrix $P \in \olt^{n \times n}$ such that $A = BP$. Let $P' \in \trop^{n \times n}$ be
obtained from $P$ by replacing every $\infty$ entry with $0$ (or indeed
any other element of $\trop$). Now for $1 \leq i, j \leq n$ we have
\begin{equation}\label{eq_z}
A_{ij} = (BP)_{ij} = \bigoplus_{k=1}^n B_{ik} P_{kj}.
\end{equation}
Since no $A_{ij}$ is $\infty$, if any $k$ and $j$ are such that
$P_{kj} = \infty$ then we must have $B_{ik} = -\infty$ for all $i$.
It follows that $B_{ik} P_{kj} = -\infty = B_{ik} P'_{kj}$ for all $k$
and $j$, whence $A = BP'$ and $A \leq_\GreenR B$ in $\trop^{n \times n}$.

Again, the converse is immediate.
\end{proof}

\begin{lemma}\label{lemma_dinherit1}
Suppose $X \in \tropn$ is $\GreenR$-related to a matrix in $\ftn$,
and $\GreenL$-related to a matrix in $\ftn$. Then $X \in \ftn$.
\end{lemma}
\begin{proof}
Suppose $X \GreenR Y \in \ftn$. Then any column of $X$ containing $-\infty$
lies in $C(X)$, which by Proposition~\ref{prop_landr} is $C(Y)$. But it is
easily seen that the only column vector in $C(Y)$ containing $-\infty$ is
the zero vector, so every column of $X$ containing $-\infty$ is a column
of $-\infty$s. A dual argument, using the fact that $X$ is $\GreenL$-related
to a matrix in $\ftn$, shows that every row of $X$ containing $-\infty$ is
a row of $-\infty$s. Now if $X \notin \ftn$ then $X$ contains some entry
equal to $-\infty$, from which we may deduce that every entry of $X$ is
$-\infty$, that is, that $X$ is the zero matrix. But the zero matrix forms
an ideal, and so must lie in a $\GreenR$-class by itself, contradicting
the fact that $X \GreenR Y$.
\end{proof}

\begin{lemma}\label{lemma_dinherit2}
Suppose $X \in \oltn$ is $\GreenR$-related to a matrix in $\tropn$,
and $\GreenL$-related to a matrix in $\tropn$. Then $X \in \tropn$.
\end{lemma}
\begin{proof}
We claim first that the column space $C(X)$ of $X$ is generated by those
columns which do not contain $\infty$. Indeed, suppose $X \GreenR Y \in \tropn$.
Then by Proposition~\ref{prop_landr}, $C(X) = C(Y)$. In particular, each
column of $Y$ is a linear combination of columns of $X$. Clearly this
combination cannot a column of $X$ with an $\infty$ entry with a coefficient
other than $-\infty$, or else the column of $Y$ with contain $\infty$. Thus,
each column of $Y$ is a linear combination of those columns of $X$ which do
not contain $\infty$. But every vector in $C(X) = C(Y)$ is a linear
combination of the columns of $Y$, and hence of the columns of $X$ which
do not contain $\infty$, as required.

By a dual argument, the row space of $X$ is generated by those rows
which do not contain $\infty$.

Now suppose for a contradiction that $X \notin \tropn$, and choose some
row $i$ and column $j$ with $X_{ij} = \infty$. By the above, the $j$th
column (call it $X_j$)
can be written as a linear combination of those columns not containing
$\infty$. Clearly, one of the columns in this combination (say column $X_k$)
must have coefficient $\infty$ and a finite entry in position $i$.
Now $\infty X_k \leq X_j$, so for any $p$ such that $X_{pj} \neq \infty$ we
must have $X_{pk} = -\infty$. In particular, in any row of $X$ not containing
$\infty$, column $k$ will contain $-\infty$. Since the rows not containing
$\infty$ span the row space $R(X)$, it follows that every row vector in $R(X)$
contains $-\infty$ in column $k$. But since the rows of $X$ lie in $R(X)$,
this contradicts the fact that row $i$ of $X$ contains a finite entry in
column $k$.
\end{proof}

The preceding propositions and lemmas combine to show that many of Green's
relations in $\ftn$ and $\tropn$ are inherited from the containing semigroup
$\oltn$.

\begin{theorem}[Inheritance of Green's Relations]\label{thm_inherit}
Each of Green's pre-orders $\leq_\GreenR$, $\leq_\GreenL$ and
equivalence relations $\GreenL$, $\GreenR$, $\GreenH$, $\GreenD$
in $\ft^{n \times n}$ or $\trop^{n \times n}$ is the restriction of
the corresponding relation in $\oltn$.
\end{theorem}
\begin{proof}
The results for $\GreenR$ and $\GreenL$ follow immediately from
Propositions~\ref{prop_rorderinherit1} and \ref{prop_rorderinherit2} and
their duals. The claim for $\GreenH$ is immediate from the claims for
$\GreenR$ and $\GreenL$. The claim for $\GreenD$ is a consequence of the
claims for $\GreenR$ and $\GreenL$ together with 
Lemma~\ref{lemma_dinherit1} and Lemma~\ref{lemma_dinherit2}.
\end{proof}

We also have the following immedate corollary of Lemmas~\ref{lemma_dinherit1}
and \ref{lemma_dinherit2}.

\begin{corollary}
$\ftn$ is a union of $\GreenH$-classes in $\tropn$ and in $\oltn$, while
$\tropn$ is a union of $\GreenH$-classes in $\oltn$.
\end{corollary}

\section{Converse Duality}\label{sec_converseduality}

In this section we shall establish a converse to Theorem~\ref{thm_duality}
in the finitary case,
showing that an anti-isomorphism between two convex sets $X$ and $Y$ in
$\ft^n$ is a
sufficient, as well as a necessary, condition for the existence of a matrix
with row space $X$ and column space $Y$. As well as being of interest in
its own right, this together with Theorem~\ref{thm_inherit} will allow us to completely describe Green's $\GreenD$
relation in finite dimensional full matrix semigroups over the tropical
semirings $\ft$, $\trop$ and $\olt$. We begin with some lemmas.

\begin{lemma}\label{lemma_kernel}
Let $S = \ft$ or $S = \olt$. 
Suppose $B \in S^{m \times n}$ is a tropical matrix and
$z \in S^{1 \times n}$ is a 
row vector not in $R_S(B)$. Then there exist column vectors
$x, y \in S^{n \times 1}$
such that $Bx=By$ but $zx \neq zy$.
\end{lemma}
\begin{proof}
Set $x = (-z)^T$, and consider the vector $Bx$, which clearly
lies in the column space of $B$. By Proposition~\ref{prop_dualitybijection}
the map $\theta_B$ is a bijection from the $R_S(B)$ to $C_S(B)$, so there
is a $v \in R_S(B)$ such that
$\theta_B(v) = Bx$. Note that $v \neq z$ since $z$ does not lie in
$R_S(B)$.
If we set $y = (-v)^T$ then by the definition of $\theta_B$ we have
$By = B(-v)^T = \theta_B(v) = Bx$, so it will suffice to show that
$zx \neq zy$. 

To this end, consider the matrix
$$C = \begin{pmatrix} z \\ B \end{pmatrix} \in \olt^{(m+1) \times n}.$$
Then $z$ (which is a row of $C$) and $v$ (which
was chosen to lie in $R_S(B)$) both lie in $R_S(C)$. Consider now the duality
map $\theta_C$. By Proposition~\ref{prop_dualitybijection} again, $\theta_C$
is injective on $R_S(C)$, so we have
$$Cx = C(-z)^T = \theta_C(z) \neq \theta_C(v) = C(-v)^T = Cy.$$
But
$$Cx = \begin{pmatrix} z \\ B \end{pmatrix} x = \begin{pmatrix} zx \\ Bx \end{pmatrix} \text{ and } Cy = \begin{pmatrix} z \\ B \end{pmatrix} y = \begin{pmatrix} zy \\ By \end{pmatrix} $$
and we know that $Bx = By$, so for $Cx \neq Cy$ we must have $zx \neq zy$.
\end{proof}

\begin{theorem}\label{thm_lorderkernel}
Let $S = \ft$ or $S = \olt$, and let $A, B \in S^{m \times n}$. Then
the following are equivalent:
\begin{itemize}
\item[(i)] $R_S(A) \subseteq R_S(B)$.
\item[(ii)] there is a linear morphism from $C_S(B)$ to $C_S(A)$ taking
the $i$th column of $B$ to the $i$th column of $A$ for all $i$.
\item[(iii)] there is a surjective linear morphism from $C_S(B)$ to $C_S(A)$ taking
the $i$th column of $B$ to the $i$th column of $A$ for all $i$.
\end{itemize}
\end{theorem}
\begin{proof}
First note that (iii) implies (ii) trivially, while if (ii) holds then the
given morphism has image including the columns of $A$, and hence
contains $C_S(A)$, and thus is surjective, so (iii) holds.

Now let $c_1, \dots, c_n$ denote the columns of $A$ and $d_1, \dots, d_n$ denote
the columns of $B$.

Suppose for a contradiction that (ii) holds and (i) does not.
Then we may choose $z \in R_S(A)$ (say $z = z'A$) such that
$z \notin R_S(B)$. Now by Lemma~\ref{lemma_kernel}, there are vectors $x$ and $y$ such that
$Bx = By$
but $zx \neq zy$. It follows from the latter that $Ax \neq Ay$, since
otherwise we would have $zx = z'Ax = z'Ay = zy$.
Now by the definition of matrix multiplication we have
$$\bigoplus_{i=1}^n{x_i c_i} = A x \neq A y = \bigoplus_{i=1}^n{y_i c_i}$$
while
$$\bigoplus_{i=1}^n{x_i d_i} = B x = B y = \bigoplus_{i=1}^n{y_i d_i},$$
which clearly contradicts the assumption that the map taking $d_i$ to
$c_i$ is a morphism of semimodules.

Conversely, suppose (i) holds. To show that (ii) holds it clearly suffices
to show that every linear relation between the columns of $B$ also holds
between the columns of $A$. Indeed, suppose
$$\bigoplus_{i=1}^n{x_i c_i} = \bigoplus_{i=1}^n{y_i c_i}$$
is a relation which holds between the columns $c_i$ of $A$.
Then letting $x$ and $y$ be the column vectors formed from the $x_i$s,
by the definition of matrix multiplication we have $Bx = By$. It follows
that $bx = by$ for every row $b$ of $B$, and hence by distributivity
for every vector in $R_S(B)$.
In particular, $bx = by$ for every vector in $R_S(A) \subseteq R_S(B)$, so
that $Ax = Ay$ and
$$\bigoplus{x_i d_i} = \bigoplus{y_i d_i}$$
as required.
\end{proof}

\begin{corollary}\label{cor_lorderkernel}
Let $S = \olt$ or $S = \ft$, and let $A, B \in S^{m \times n}$.
Then $R_S(A) = R_S(B)$ if and only if there is a linear isomorphism
from $C_S(A)$ to $C_S(B)$ taking the $i$th column of $B$ to the $i$th
column of $A$ for all $i$.
\end{corollary}
\begin{proof}
If $R_S(A) = R_S(B)$ then $R_S(A) \subseteq R_S(B)$ and $R_S(B) \subseteq R_S(A)$, so
by applying Theorem~\ref{thm_lorderkernel} twice 
there is a surjective
morphism from $C_S(B)$ to $C_S(A)$ taking the columns of $B$ to the respective
columsn of $A$, and a surjective morphism from $C_S(A)$ to $C_S(B)$ taking the
columns of $A$ to the respective columns of $C_S(B)$. Since these maps are
mutually inverse on the columns, which are generating sets for the respective
matrices, it is immediate that they are mutually inverse maps from $C_S(A)$
to $C_S(B)$, and hence must be isomorphisms.

Conversely, if $f : C_S(A) \to C_S(B)$ is an isomorphism taking the columns of $A$
to the respective columns of $B$, then its inverse is a morphism taking
the columns of $B$ to the respective columns of $A$. Applying
Theorem~\ref{thm_lorderkernel} to each of these functions we obtain
$R_S(A) \subseteq R_S(B)$ and $R_S(B) \subseteq R_S(A)$.
\end{proof}

The above results allow us to establish our promised
converse to the duality theorem (Theorem~\ref{thm_duality} above).

\begin{theorem}[Exact Duality Theorem]\label{thm_dgeometry}
Let $S = \ft$ or $S = \olt$. Suppose $X$ be an $m$-generated convex subset
of $S^n$ and $Y$ is an $n$-generated convex subset of $S^m$.
Then $X$ and $Y$ are anti-isomorphic if and only if there is a matrix
$M \in S^{m \times n}$ with $R_S(M) = X$ and $C_S(M) = Y$.
\end{theorem}
\begin{proof}
Suppose $X$ and $Y$ are anti-isomorphic. Choose two $m \times n$ matrices $A$ and $B$ such that $A$ has row space
$X$ and $B$ has column space $Y$. Then by the dual to Theorem~\ref{thm_duality}, there
is an anti-isomorphism from $C_S(A)$ to $R_S(A) = X$. By Lemma~\ref{lemma_composeantis},
composing with the anti-isomorphism from $X$ and $Y = C_S(B)$ we may thus obtain an
isomorphism $f : C_S(A) \to C_S(B)$.
 Let $D$ be the $m \times n$ matrix whose $i$th column
is the image under $f$ of the $i$th column of $A$. Then by
Corollary~\ref{cor_lorderkernel} we have $R_S(A) = R_S(D)$. Also, since $f$ is
an isomorphism, the image under $f$ of a generating set for $C_S(A)$ must be
a generating set for $C_S(B)$. In particular, the columns of $D$ are a
generating set for $C_S(B)$, that is, $C(D) = C(B) = Y$. Now by
Corollary~\ref{thm_lorderkernel}
to $f$ and its inverse, we have $R_S(D) \subseteq R_S(A) = X$ and
$X = R_S(A) \subseteq R_S(D)$. Thus, the matrix $D$ has the required properties.

The converse is Theorem~\ref{thm_duality}.
\end{proof}

\section{The $\GreenD$ Relation}\label{sec_greend}

From Theorem~\ref{thm_dgeometry} and Lemma~\ref{lemma_composeantis} we obtain
a number of equivalence geometric characterisations of Green's $\GreenD$
relation in $\ft$ and $\olt$.

\begin{theorem}[Green's $\GreenD$ Relation for $\ft^{n \times n}$ and $\olt^{n \times n}$]\label{thm_d}
Let $S = \ft$ or $S = \olt$, and let $A$ and $B$ be matrices in
$S^{n \times n}$. Then the
following are equivalent:
\begin{itemize}
\item[(i)] $A \GreenD B$ in $S^{n \times n}$;
\item[(ii)] $C_S(A)$ and $C_S(B)$ are isomorphic as semimodules;
\item[(iii)] $R_S(A)$ and $R_S(B)$ are isomorphic as semimodules;
\item[(iv)] $C_S(A)$ and $R_S(B)$ are anti-isomorphic as semimodules;
\item[(v)] $R_S(A)$ and $C_S(B)$ are anti-isomorphic as semimodules.
\end{itemize}
\end{theorem}
\begin{proof}
First suppose (i) holds. Then by definition there exists a matrix $D$
such that $A \GreenL D \GreenR B$. By Proposition~\ref{prop_landr}, we
have $R_S(A) = R_S(D)$ and $C_S(D) = C_S(B)$. But by Theorem~\ref{thm_dgeometry},
$R_S(D)$ and $C_S(D)$ are anti-isomorphic, so $R_S(A)$ and $C_S(B)$ are anti-isomorphic,
and so (v) holds. A dual argument shows that (i) implies (iv).

Next suppose (v) holds. By Theorem~\ref{thm_dgeometry}, there is an
anti-isomorphism from $C_S(A)$ to $R_S(A)$. By Lemma~\ref{lemma_composeantis}
this composes with the anti-isomorphism from $R_S(A)$ to $C_S(B)$ to produce
an isomorphism between $C_S(A)$ and $C_S(B)$, so that (ii) holds. Similar
arguments establish that (v) implies (iii), (iv) implies (iii) and (iv)
implies (ii).

Finally, suppose (ii) holds, and let $f : C_S(A) \to C_S(B)$ be an isomorphism.
Let $D$ be the matrix obtained from $A$ by applying $f$ to each column. Then
by Corollary~\ref{cor_lorderkernel}, we have $R_S(A) = R_S(D)$ so that $A \GreenL D$.
Moreover, since the isomorphism $f$ must map a generating set for $C_S(A)$ to a
generating set for $C_S(B)$, the columns of $D$ form a generating set for $C_S(D)$,
that is, $C_S(D) = C_S(B)$, so $D \GreenR B$. Thus, $A \GreenD B$ and (i) holds.
\end{proof}

Theorem~\ref{thm_d} and Theorem~\ref{thm_inherit} together yield a
description of $\GreenD$ for matrices over $\trop$ in terms of their
$\olt$-linear column or row spaces. It is natural to ask also whether
$\GreenD$ can be characterised in terms of $\trop$-linear column and row
spaces.

\begin{lemma}\label{lemma_tropolt}
Let $X$ be a convex subset of $\trop^n$ and $X'$ be the convex subset of
$\olt^n$ which it generates. Then for any $x \in X'$ the following are equivalent.
\begin{itemize}
\item[(i)] $x \notin X$;
\item[(ii)] $x$ contains $\infty$ in some component;
\item[(iii)] $x = \infty a \oplus b$ for some $a, b \in X'$ with $a$ not the zero vector; 
\item[(iv)] $x = \infty a \oplus b$ for some $a, b \in X$ with $a$ not the zero vector;
\end{itemize}
\end{lemma}
\begin{proof}
Suppose (i) holds, that is, that $x \notin X$. Since $x \in X'$, it may be
written as a $\olt$-linear combination of finitely many vectors in $X$. Using
distributivity and commutativity to collect together the terms with coefficient
$\infty$ and the terms with other coefficients, we may thus write
$x = \infty a \oplus b$ where $a$ is a sum of vectors in $X$ (and hence lies
in $X$), and $b$ is a $\trop$-linear combination of vectors in $X$ (and hence
lies in $X$). Finally, if $a$ were the zero vector then we would have
$x = b \in X$ giving a contradiction. Thus, (iv) holds.

That (iv) implies (iii) is immediate. If (iii) holds then since $a$ is not
the zero vector, $\infty a$ contains $\infty$ in some component, so
$x = \infty a \oplus b$ contains $\infty$ in some component, and (ii) holds.
Finally, that (ii) implies (i) is obvious.
\end{proof}

\begin{lemma}\label{lemma_welldef}
Let $X$ be a convex subset of $\trop^n$ and $X'$ be the convex subset of
$\olt^n$ which it generates. Then for any $a, b, a', b' \in X$ we have
that $\infty a \oplus b = \infty a' \oplus b'$ if and only if
$d_H(a,a') \neq \infty$ and $b \oplus \lambda a = b' \oplus \lambda a$
for all sufficiently large $\lambda$.
\end{lemma}
\begin{proof}
Suppose $\infty a \oplus b = \infty a' \oplus b'$. Since $a, b \in \trop^n$, they
do not contain any $\infty$ positions. It follows that
$\infty a \oplus b$ contains an $\infty$ in position $i$ exactly if $a$ does
\textbf{not} contain $-\infty$ in this position. By symmetry of assumption
this is true exactly if $a'$ does \textbf{not} contain $-\infty$ in this
position. Thus, $a$ and $a'$ contain $-\infty$ in exactly the same positions.
Since neither contains $\infty$, this means that $d_H(a,a') \neq \infty$.
Notice also that in any position where $a$ contains $-\infty$, the expression
$\infty a \oplus b$ takes the value of $b$ and hence by symmetry also of $b'$.
Thus, $b$ and $b'$ agree in such positions.
Hence, if we choose $\lambda$ large enough that $\lambda a$ exceeds $b$ and
$b'$ in all positions where $a$ is not $-\infty$, then we obtain
$b \oplus \lambda a = b' \oplus \lambda a$.

Conversely, suppose $d_H(a,a') \neq \infty$ and
$b \oplus \lambda a = b' \oplus \lambda a$ for all sufficiently large
$\lambda$. Then $a$ and $a'$ have $-\infty$ in the same positions, from
which it follows that $\infty a = \infty a'$. Moreover, since
$b \oplus \lambda a = b' \oplus \lambda a$ it is easy to see that
$b$ and $b'$ agree in every position where $a$ takes the value $-\infty$,
from which it follows that $\infty a \oplus b = \infty a \oplus b' = \infty a' \oplus b'$.
\end{proof}

\begin{theorem}[Inheritance and Extension of Isomorphisms]\label{thm_tropoltiso}
Let $X$ and $Y$ be convex subsets of $\trop^i$ and $\trop^j$ respectively,
and let $X'$ and $Y'$ be the convex subsets of $\olt^i$ and $\olt^j$ which
they generate. Then $X$ and $Y$ are isomorphic (as semimodules over $\trop$)
if and only if $X'$ and $Y'$ are isomorphic (as semimodules over $\olt$).
\end{theorem}
\begin{proof}
Suppose first that $f : X' \to Y'$ is an isomorphism. We claim that $f$
sends elements of $X$ to elements of $Y$. Indeed, suppose $x \in X$. Then
by Lemma~\ref{lemma_tropolt}, $x$ cannot be written in the form $a \infty \oplus b$
for any $a, b \in X'$ with $a$ not the zero vector. Since $f$ is an
isomorphism (and in particular
preserves the zero vector) it follows that $f(x)$ cannot be written as
$\infty c \oplus d$ for any $c, d \in Y'$ with $c$ not the zero vector. Thus,
by Lemma~\ref{lemma_tropolt} again, $f(x)$ lies in $Y$. A similar argument
shows that the inverse of $f$ maps $Y$ into $X$, and it follows that $f$
restricts to an isomorphism of $X$ to $Y$.

Conversely, suppose that $g : X \to Y$ is an isomorphism. We claim that
$g$ admits an extension to $X'$ well defined by:
$$\hat{g}(\infty a \oplus b) = \infty g(a) \oplus g(b).$$
To show that this is well defined, suppose $\infty a \oplus b = \infty a' \oplus b'$.
Then by Lemma~\ref{lemma_welldef} we have
$d_H(a,a') \neq \infty$ and $b \oplus \lambda a = b' \oplus \lambda a$
for all sufficiently large $\lambda$.
Using the fact that $g$ is an isomorphism (and in particular preserves
the Hilbert metric) we have
$d_H(g(a),g(a')) \neq \infty$ and $g(b) \oplus \lambda g(a) = g(b') \oplus \lambda a$
for all sufficiently large $\lambda$. Now by Lemma~\ref{lemma_welldef} again,
$\infty g(a) \oplus g(b) = \infty g(a') \oplus g(b')$, as required to show that
$\hat{g}$ is well-defined.

Next we claim that $\hat{g}$ is linear. The fact that $\hat{g}$ respects
addition and scaling by elements of $\trop$ follows immediately from the
definition and the elementary properties of the semiring $\olt$. It remains
to show that $\hat{g}$ respects scaling by $\infty$. Let $x \in X'$. Then
$x$ can be written as $\infty a \oplus b$ for some $a,b \in X$, and we have
\begin{align*}
\hat{g}(\infty x) &= \hat{g}(\infty(\infty a \oplus b)) = \hat{g}(\infty a \oplus \infty b) = \infty g(a \oplus b) = \infty g(a) \oplus \infty g(b) \\
&= \infty \infty g(a) \oplus \infty g(b) = \infty (\infty g(a) \oplus g(b)) = \infty \hat{g}(\infty a \oplus b) = \infty \hat{g}(x).
\end{align*}

Now if $h : Y \to X$ is the inverse of $g$ then the same argument shows that
$h$ extends to a linear map $\hat{h} : Y' \to X'$ satisfying $\hat{h}(\infty c \oplus d) = \infty h(c) \oplus h(d)$
for all $c, d \in Y$. Thus, for any $x \in X'$ we have $x = a \infty \oplus b$ for
some $a, b \in X$, whereupon
$$\hat{h}(\hat{g}(x)) = \hat{h}(\hat{g}(\infty a \oplus b)) = \hat{h}(\infty g(a) \oplus g(b)) = \infty h(g(a)) \oplus h(g(b)) = \infty a \oplus b = x.$$
By the same argument we have $\hat{g}(\hat{h}(y)) = y$ for all $y \in Y'$, so that
$\hat{h}$ is an inverse for $\hat{g}$. Thus, $\hat{g}$ is an isomorphism
from $X$ to $Y$.
\end{proof}

Combining Theorem~\ref{thm_tropoltiso} with Theorem~\ref{thm_d}, we obtain
an additional description of the $\GreenD$ relation for full matrix
semigroups over $\trop$.
\begin{theorem}[Green's $\GreenD$ Relation for $\trop^{n \times n}$]\label{thm_dtrop}
Let $A, B \in \tropn$. Then the following are equivalent:
\begin{itemize}
\item[(i)] $A \GreenD B$ in $\tropn$;
\item[(ii)] $A \GreenD B$ in $\oltn$ (and the other four equivalent conditions given by \\ Theorem~\ref{thm_d} in the case $S = \olt$);
\item[(iii)] $C_\trop(A)$ and $C_\trop(B)$ are isomorphic;
\item[(iv)] $R_\trop(A)$ and $R_\trop(B)$ are isomorphic;
\end{itemize}
\end{theorem}
\begin{proof}
The equivalence of (i) and (ii) is part of Theorem~\ref{thm_inherit}. 
By Theorem~\ref{thm_d}, (ii) is equivalent to the statement that
$C_\olt(A)$ and $C_\olt(B)$ are isomorphic. But $C_\olt(A)$ [respectively,
$C_\olt(B)$] is generated as a semimodule over $\olt$ by the columns of $A$ [$B$], and hence by
$C_\trop(A)$ [$C_\trop(B)$]. Hence, by Theorem~\ref{thm_tropoltiso}, (ii) is
equivalent to (iii). The equivalence of (ii) and (iv) is established by a
dual argument.
\end{proof}

\section{Remarks}\label{sec_remarks}

We remark briefly on the extent to which our algebraic results, and in
particular Theorem~\ref{thm_d}, might apply in wider contexts. Considering
Theorem~\ref{thm_d}, we note that while the
equivalence of (i), (iv) and (v) is closely bound up with matrix duality, conditions (i), (ii) and (iii)
can be shown directly to be equivalent without explicit recourse to
duality, by using Theorem~\ref{thm_lorderkernel} and Corollary~\ref{cor_lorderkernel}. These results depend
essentially only upon Lemma~\ref{lemma_kernel}. While we proved this lemma
using matrix duality, it is likely that an appropriate analogues hold in other
semirings for different reasons. The conditions (i), (ii) and (iii) are
equivalent, and hence yield characterisations of $\GreenD$ in terms
of the isomorphisms of row spaces and isomorphisms of column spaces, for
matrices over any such semiring. More generally, we believe that semirings
satisfying the condition given in the tropical case by Lemma~\ref{lemma_kernel}
are likely to form a ``well-behaved'' class, encompassing many examples of
interest. As such, they may be deserving of axiomatic study.

Since our methods do not essentially depend upon the
matrices considered being square, similar methods should yield corresponding
results for Green's relations in the small categories of all finite
dimensional matrices (not necessarily square or of uniform size) over
$\ft$, $\trop$ and $\olt$ respectively.

Finally, we note that the $\GreenJ$ relation and the $\leq_\GreenJ$
pre-order for tropical matrix semigroups remain poorly understood, and are
deserving of further study.

\section*{Acknowledgements}

This research was supported by EPSRC grant number EP/H000801/1
(\textit{Multiplicative Structure of Tropical Matrix Algebra}).
The second author's research is also supported by an RCUK
Academic Fellowship. The authors thank Zur Izhakian, Marianne Johnson,
Stuart Margolis and Sergei Sergeev for helpful conversations.

\bibliographystyle{plain}

\def\cprime{$'$} \def\cprime{$'$}

\end{document}